\documentclass[12pt]{amsart}
\usepackage{amscd,amssymb,amsmath,multicol}
\newtheorem{thm}[equation]{Theorem}
\numberwithin{equation}{section}
\newtheorem{cor}[equation]{Corollary}

\newtheorem{rmk}[equation]{Remark}

\newtheorem{defin}[equation]{Definition}
\newtheorem{diag}[equation]{Diagram}
\newtheorem{prop}[equation]{Proposition}
\newtheorem{tab}[equation]{Table}

\begin{document}
\raggedbottom \voffset=-.7truein \hoffset=0truein \vsize=8truein
\hsize=6truein \textheight=8truein \textwidth=6truein
\baselineskip=18truept
\def\vareps{\varepsilon}
\def\mapright#1{\ \smash{\mathop{\longrightarrow}\limits^{#1}}\ }
\def\mapleft#1{\smash{\mathop{\longleftarrow}\limits^{#1}}}
\def\mapup#1{\Big\uparrow\rlap{$\vcenter {\hbox {$#1$}}$}}
\def\mapdown#1{\Big\downarrow\rlap{$\vcenter {\hbox {$\ssize{#1}$}}$}}
\def\on{\operatorname}
\def\spa{\on{span}}
\def\a{\alpha}
\def\bz{{\Bbb Z}}
\def\gd{\on{gd}}
\def\imm{\on{imm}}
\def\sq{\on{Sq}}
\def\ssp{\on{stablespan}}
\def\eps{\epsilon}
\def\br{{\Bbb R}}
\def\bc{{\Bbb C}}
\def\bh{{\Bbb H}}
\def\tfrac{\textstyle\frac}
\def\w{\wedge}
\def\b{\beta}
\def\A{{\cal A}}
\def\P{{\cal P}}
\def\zt{{\Bbb Z}_2}
\def\bq{{\Bbb Q}}
\def\ker{\on{ker}}
\def\coker{\on{coker}}
\def\u{{\cal U}}
\def\e{{\cal E}}
\def\exp{\on{exp}}
\def\wbar{{\overline w}}
\def\xbar{{\overline x}}
\def\ybar{{\overline y}}
\def\zbar{{\overline z}}
\def\ebar{{\overline e}}
\def\nbar{{\overline n}}
\def\mbar{{\overline m}}
\def\ubar{{\overline u}}
\def\et{{\widetilde E}}
\def\ni{\noindent}
\def\coef{\on{coef}}
\def\den{\on{den}}
\def\gd{{\on{gd}}}
\def\N{{\Bbb N}}
\def\Z{{\Bbb Z}}
\def\Q{{\Bbb Q}}
\def\R{{\Bbb R}}
\def\C{{\Bbb C}}
\def\Bin{\on{Bin}}
\title[Projective product spaces]
{Projective product spaces}
\author{Donald M. Davis}
\address{Department of Mathematics, Lehigh University\\Bethlehem, PA 18015, USA}
\email{dmd1@lehigh.edu}
\date{August 7, 2009}


\maketitle
\begin{abstract} Let $\nbar=(n_1,\ldots,n_r)$.  The quotient space $P_\nbar:=S^{n_1}\times\cdots\times S^{n_r}/(\xbar\sim-\xbar)$ is what we call a projective product space.
We determine the integral cohomology ring $H^*(P_\nbar)$ and the action of the Steenrod algebra on $H^*(P_\nbar;\zt)$. We give  a splitting of $\Sigma P_\nbar$ in terms of stunted real projective spaces, and determine when $S^{n_i}$ is a product factor of $P_\nbar$. We relate the immersion dimension and span of $P_\nbar$ to the
much-studied sectioning question for multiples of the Hopf bundle over real projective spaces.
We show that the immersion dimension of $P_\nbar$ depends only on $\min(n_i)$, $\sum n_i$, and $r$, and determine its precise value unless all $n_i\ge10$.
We also determine exactly when $P_\nbar$ is parallelizable.
 \end{abstract}

\section{Introduction}\label{intro} If $\nbar=(n_1,\ldots,n_r)$ with $n_i$ positive integers, let
$$P_{\nbar}=S^{n_1}\times\cdots\times S^{n_r}/((x_1,\ldots,x_r)\sim(-x_1,\ldots,-x_r)),$$
where $x_i\in S^{n_i}$.
This is a manifold of dimension $|\nbar|:=n_1+\cdots+n_r$, which we call a projective product space. If $\nbar=(n)$, then $P_\nbar=P^n$, the real
projective space.

There is a 1-dimensional vector bundle $\xi_\nbar$ over $P_\nbar$ for which the $k$-fold Whitney sum $k\xi_\nbar$ has total space
$$S^{n_1}\times\cdots\times S^{n_r}\times\br^k/((x_1,\ldots,x_r,t_1,\ldots,t_k)\sim(-x_1,\ldots,-x_r,-t_1,\ldots,-t_k)),$$
and its sphere bundle clearly satisfies
\begin{equation}\label{sph}S(k\xi_\nbar)=P_{(n_1,\ldots,n_r,k-1)}.\end{equation}
Thus each space $P_\nbar$ can be built up  iteratively as sphere bundles. For example,
$P_{(n_1,n_2,n_3)}=S((n_3+1)\xi_{(n_1,n_2)})$ with $\xi_{(n_1,n_2)}$ the line bundle over
$P_{(n_1,n_2)}=S((n_2+1)\xi_{n_1})$.

In Section \ref{cohsec}, we determine the $A$-algebra $H^*(P_\nbar;\zt)$ (Theorem \ref{cohthm}) and the algebra $H^*(P_\nbar;\bz)$ (Theorem \ref{intcoh}), and determine in Theorem \ref{Kthm} the ring $K^*(P_\nbar)$.
We show that $\Sigma P_\nbar$ splits as a wedge of desuspensions of stunted
projective spaces (Theorem \ref{splitthm}), and determine when $S^{n_i}$ is a product factor of $P_\nbar$ (Theorem \ref{he}).   In Section \ref{mfsec}, we relate the immersion dimension (Theorem \ref{immthm}) and span (Theorem \ref{spanthm}) of $P_\nbar$ to results about sectioning
multiples of the Hopf bundles over projective spaces, and we determine exactly when $P_\nbar$ is parallelizable (Theorem \ref{pithm}). In Section \ref{numsec}, we use known results for projective spaces to give some
numerical results for the immersion dimension and span of $P_\nbar$. We give the precise value of the immersion dimension of $P_\nbar$ unless
all $n_i\ge10$.

\section{Cohomology of $P_\nbar$ and a splitting}\label{cohsec}

The first  property of the spaces $P_\nbar$ which we study is
their mod-2 cohomology. Here and throughout, $\Lambda(-)$ denotes an exterior algebra,  $\zt=\bz/2$,  $A$ is the mod 2 Steenrod algebra,
and $\sq=\sum\limits_{n\ge0}\sq^n$.
\begin{thm}\label{cohthm} Let $\nbar=(n_1,\ldots,n_r)$ with $n_1\le n_2\le\cdots\le n_r$.
If $n_1<n_2$, or $n_1$ is odd, then there is an isomorphism of $A$-algebras
$$H^*(P_\nbar;\zt)\approx\zt[y]/y^{n_1+1}\otimes\Lambda[x_{n_2},\ldots,x_{n_r}]$$
with $|x_{n_i}|=n_i$, $|y|=1$, $\sq(x_{n_i})=x_{n_i}(1+y)^{n_i+1}$, and  $\sq(y)=y(1+y)$.
If $n_1$ is even and $n_1=\cdots=n_k<n_{k+1}$ for some $k>1$, then $H^*(P_\nbar;\zt)$ is as above, except that
$x_{n_i}^2=y^{n_1}x_{n_i}$ for $2\le i\le k$.
\end{thm}
\begin{proof} The proof proceeds by induction on $r$, with the case $r=1$ being the well-known result
for $P^{n_1}$. Let $\mbar=(n_1,\ldots,n_{r-1})$, and assume the result known for $P_\mbar$. Since $P_\nbar\approx S((n_r+1)\xi_\mbar)$, there is a cofibration
\begin{equation}\label{cofibr}P_\nbar\mapright{p} P_\mbar \mapright{i} T((n_r+1)\xi_\mbar),\end{equation}
where $T(-)$ denotes the Thom space, and
$$p([x_1,\ldots,x_r])=[x_1,\ldots,x_{r-1}]$$
for $x_i\in S^{n_i}$. Hence there is an exact sequence, with  coefficients always in $\zt$,
\begin{equation}\label{LES}H^{*+1}(T((n_r+1)\xi_\mbar))\ \mapleft{\delta}\  H^*(P_\nbar)\ \mapleft{p^*}\ H^*(P_\mbar)\ \mapleft{i^*}\  H^*(T((n_r+1)\xi_\mbar)).\end{equation}
Since $n_{r-1}\le n_r$, there is a map $P_\mbar\mapright{j} P_\nbar$ defined by $$j([x_1,\ldots,x_{r-1}])=[x_1,\ldots,x_{r-1},x_{r-1}].$$
In the last component here, $S^{n_{r-1}}$ is identified as the obvious subspace of $S^{n_r}$. Since $p\circ j$ is the identity map of $P_{\mbar}$,
 (\ref{LES}) splits, yielding
\begin{equation}\label{split}H^*(P_\nbar)\approx H^*(P_\mbar)\oplus H^{*+1}(T((n_r+1)\xi_\mbar)),\end{equation}
and the splitting is as $A$-modules.

Let $x_{n_r}\in H^{n_r}(P_\nbar)$ correspond to the Thom class $U\in H^{n_r+1}(T(n_r+1)\xi_\mbar)$ under (\ref{split}).
Then $\sq(x_{n_r})$ corresponds to
$$\sq(U)=W((n_r+1)\xi_\mbar)U=(1+y)^{n_r+1}U.$$
Here $\sq$ and $W$ are the total Steenrod square and the total Stiefel-Whitney class, and we have used that
the projection $P_\mbar\mapright{p} P^{n_1}$ has $p^*(\xi_{n_1})=\xi_\mbar$.

By the Thom isomorphism, the second summand of the RHS of (\ref{split}) is $H^*(P_\mbar)\cdot x_{n_r}$.
Thus (\ref{split}) becomes
$$H^*(P_\nbar)\approx H^*(P_\mbar)\oplus H^*(P_\mbar)\cdot x_{n_r},$$
and the isomorphism is as rings, using the multiplication of $H^*(P_\mbar)$ on the RHS.
 Finally, \begin{equation}\label{sq}x_{n_r}^2=\sq^{n_r}x_{n_r}=\tbinom{n_r+1}{n_r}y^{n_r}x_{n_r},\end{equation}
which is 0 if $n_1<n_r$ or $n_r$ is odd. The case described in the last sentence of the theorem is also
immediate from (\ref{sq}). Thus the induction is extended.
\end{proof}

\begin{rmk} {\rm We can give an explicit formula for the map $\Sigma P_\nbar\to T((n_r+1)\xi_\mbar)$ which splits
the cofibration (\ref{cofibr}).  If $\xbar\in S^{n_1}\times\cdots\times S^{n_{r-2}}$, then our map sends
\begin{equation}\label{explicit}[t,\xbar,x_{n_{r-1}},x_{n_r}]\mapsto[\xbar,x_{n_{r-1}},tx_{n_{r-1}}+(1-t)x_{n_r}].
\end{equation}
Here $t\in[0,1]$, and $tx_{n_{r-1}}+(1-t)x_{n_r}$ takes place in $D^{n_r+1}$, the fiber of the disk bundle.
It is a path between two points of the sphere bundle, which are identified to the basepoint in the Thom space.}\end{rmk}

We also need the following result about $H^*(P_\nbar;\bq)$ and $H^*(P_\nbar;\bz/p)$ with $p$ an odd prime. To set notation, recall that $H^*(S^{n_1}\times\cdots\times S^{n_r};F)$
is the exterior algebra over $F$ on classes $x_{n_i}$, $1\le i\le r$, with $|x_{n_i}|=n_i$.
\begin{thm} \label{fldcoef} Let $F=\bq$ or $\bz/p$ with $p$ an odd prime. The homomorphism
$$H^*(P_\nbar;F)\mapright{\rho^*} H^*(S^{n_1}\times\cdots\times S^{n_r};F)$$
induced by the quotient map $\rho$ sends $H^*(P_\nbar;F)$ isomorphically to the $F$-span of all products
$x_{n_{i_1}}\cdots x_{n_{i_k}}$ such that $\sum\limits_{j=1}^k (n_{i_j}+1)$ is even.\end{thm}
\begin{proof} Let $S_\nbar=S^{n_1}\times\cdots\times S^{n_r}$. The map $S_\nbar\mapright{\rho} P_\nbar$ is a double cover, and so there is
a fibration $S_\nbar\to P_\nbar\to K(\zt,1)=RP^\infty$. This has a Serre spectral sequence with local coefficients
$$E_2^{p,q}=H^p(RP^\infty;{\mathcal H}^q(S_\nbar;F))\Rightarrow H^*(P_\nbar;F).$$
The action of the generator of $\pi_1(RP^\infty)$ on $x_{n_{i_1}}\cdots x_{n_{i_k}}$ is by multiplication by $\prod(-1)^{n_{i_j}+1}$.
Classes in ${\mathcal H}^q(S_\nbar;F)$ with trivial action will yield $F$ in $E_2^{0,q}$ and nothing else, while those with nontrivial action $\phi$
yield nothing at all in $E_2$. This latter can be seen by noting, for example from \cite[p.100]{Green}, that $H^*(RP^\infty;\bz_\phi)$
is $\zt$ for odd positive values of $*$, and 0 for even values of $*$, including $*=0$. By the Universal Coefficient Theorem, if $\bz$ is replaced by
$F$, all groups become 0. Indeed, it is the cohomology of a cochain complex with $F$ in each nonnegative grading and $\delta=2:C_{2i}\to C_{2i+1}$,
$i\ge0$.
Thus the spectral sequence collapses, having as its only nonzero summands $F$ generated by $x_{n_{i_1}}\cdots x_{n_{i_k}}$
in $E_2^{0,q}$ with $q=\sum n_{i_j}$ whenever $\sum (n_{i_j}+1)$ is even.
\end{proof}

Our next result is a splitting of $\Sigma P_\nbar$ as a wedge of desuspensions of stunted projective spaces. Here $P_n^k=RP^k/RP^{n-1}$. This splitting
has the potential to be used in studying $\spa(P_\nbar)$ at the end of the next section.
It is also useful in analyzing $K^*(P_\nbar)$ in the proof of Theorem \ref{Kthm}.

\begin{thm}\label{splitthm} Let $\nbar=(n_1,\ldots,n_r)$ with $n_1\le n_i$ for all $i$. There is a homotopy equivalence
\begin{equation}\label{splitsig}\Sigma P_\nbar\simeq \bigvee_{\ubar\subseteq(n_2,\ldots,n_r)}\Sigma^{1-\ell(\ubar)}P_{|\ubar|+\ell(\ubar)}^{n_1+|\ubar|+\ell(\ubar)}.\end{equation}
Here, if $\ubar=(u_1,\ldots,u_s)$, then $|\ubar|:=u_1+\cdots+u_s$ and $\ell(\ubar):=s$.
\end{thm}
Here we use subset notation for ordered subsets such as $\nbar$ and $\ubar$, which may have repeated entries.
Note that there are $2^{r-1}$ wedge summands, each with $n_1+1$ cells, corresponding nicely to Theorem \ref{cohthm}, as does the $A$-action.
The desuspensions on the RHS of (\ref{splitsig}) exist for dimensional reasons.
\begin{proof} The proof is by induction on $r$. Let $n_1\le\ldots\le n_r$, and let $\mbar=(n_1,\ldots,n_{r-1})$, as in the proof of
\ref{cohthm}. Because of the map $P_\mbar\mapright{j} P_\nbar$ such that $p\circ j=1_{P_\mbar}$, $i$ is null-homotopic in (\ref{cofibr}),
and so there is a splitting
\begin{equation}\label{cofsplit}\Sigma P_\nbar\ \simeq\ \Sigma P_\mbar\ \vee\  T((n_r+1)\xi_\mbar).\end{equation}
For each summand of the RHS of (\ref{splitsig}) such that $n_r\in\ubar$, we will construct a map
\begin{equation}\label{mapf}T((n_r+1)\xi_\mbar)\longrightarrow\Sigma^{1-\ell(\ubar)}P_{|\ubar|+\ell(\ubar)}^{n_1+|\ubar|+\ell(\ubar)}
\end{equation}
such that, when preceded by the projection $\Sigma P_\nbar\to T((n_r+1)\xi_\mbar)$, the composite has cohomology homomorphism
injecting onto $\sigma(\zt[y]/y^{n_1+1})\prod\limits_{j\in\ubar}x_j\subset H^*(\Sigma P_\nbar;\zt)$. Using these and (\ref{cofsplit}) and the induction hypothesis applied to
$P_\mbar$, we obtain  maps from $\Sigma P_\nbar$ into all spaces in the wedge in (\ref{splitsig}), and hence, using the co-H-structure
of $\Sigma P_\nbar$, we obtain the desired map in (\ref{splitsig}), which induces an isomorphism in $\zt$-cohomology.

To construct (\ref{mapf}), we first construct a map \begin{equation}\label{umap}T((n_r+1)\xi_{\ubar'})\mapright{f_\ubar}\Sigma^{1-\ell(\ubar)}P_{|\ubar|+\ell(\ubar)}^{n_1+|\ubar|+\ell(\ubar)},
\end{equation}
where $\ubar'=\ubar\cup\{n_1\}-\{n_r\}$,
and then precede it by the projection $T((n_r+1)\xi_\mbar)\to T((n_r+1)\xi_{\ubar'})$.
We obtain (\ref{umap}) by constructing a map of the $(\ell(\ubar)-1)$-fold suspensions, and then noting that it desuspends to the
desired map. Here we use that a map $\Sigma^t X\to\Sigma^t Y$ is a $t$-fold suspension if the dimension of $X$ is less than twice the
dimension of the bottom cell of $Y$. (See, e.g., \cite[1.11]{Milg})

Let $\ubar=(u_1,\ldots,u_s)$ with $u_s=n_r$. The suspended version of (\ref{umap}) is a map
\begin{equation}\label{suspended}\Sigma^{s-1}T((u_s+1)\xi_{(n_1,u_1,\ldots,u_{s-1})})\to P_{u_1+\cdots+u_s+s}^{u_1+\cdots+u_s+s+n_1}\end{equation}
defined by
$$[t_1,\ldots,t_{s-1},x,y,z_1,\ldots,z_{s-1}]\mapsto[\tfrac MLx,\tfrac MLt_1z_1\ldots,\tfrac MLt_{s-1}z_{s-1},\sqrt{1-M^2}y],$$
where $t_i\in[-1,1]$, $x\in D^{u_s+1}$, $y\in S^{n_1}\subset\br^{n_1+1}$, and $z_i\in S^{u_i}\subset\br^{u_i+1}$. Also, $M=\max(|t_1|,\ldots,
|t_{s-1}|,\|x\|)$, and $L=\sqrt{t_1^2+\cdots+t_{s-1}^2+\|x\|^2}$. Since $\|z_i\|=1=\|y\|$, one easily checks that the image point
has norm 1 in $\br^{u_s+1}\times\br^{u_1+1}\times\cdots\times\br^{u_{s-1}+1}\times\br^{n_1+1}$. The map clearly respects the antipodal
action, and if any $|t_i|=1$ or $\|x\|=1$, then the image point is in the subspace $P^{u_1+\cdots+u_{s}+s-1}$, which is collapsed
in the target space. Thus the map is well-defined. One readily checks that it sends the interior of the top cell bijectively,
and so the top cohomology class maps across. The map is natural with respect to decreasing values of $n_1$, and hence induces an injection
in mod 2 cohomology, as claimed above.

Let $F=\bq$ or $\bz/p$ with $p$ an odd prime. Since $H^*(P_n^k;F)$ has $F$ in $*=n$ if $n$ is even, and in $*=k$ if $k$ is odd, and nothing
else, one readily checks, using Theorem \ref{fldcoef}, that the two spaces in (\ref{splitsig}) have isomorphic $F$-cohomology groups.
In (\ref{umap}), if $n_1+|\ubar|+\ell(\ubar)$ is odd, then, by the above observation about the top cell of the map just constructed,
the $F$-cohomology homomorphism induced by (\ref{umap}) is nontrivial in the top dimension. If $|\ubar|+\ell(\ubar)$ is even, then
the $F$-cohomology homomorphism induced by (\ref{umap}) is nontrivial in dimension $|\ubar|+1$ by consideration of the above construction
when $n_1=0$. In our inductive construction of the map from the LHS of (\ref{splitsig}) to the RHS, all summands of the map
ultimately come from (\ref{umap}). Thus the $F$-cohomology homomorphism induced by (\ref{splitsig}) is bijective. Since the map
$$\Sigma P_\nbar\to \bigvee_{\ubar\subseteq(n_2,\ldots,n_r)}\Sigma^{1-\ell(\ubar)}P_{|\ubar|+\ell(\ubar)}^{n_1+|\ubar|+\ell(\ubar)}$$
of simply-connected
spaces induces an isomorphism in $\bq$- and $\bz/p$-cohomology for all primes $p$, it is a homotopy equivalence.
\end{proof}

Next we determine the integral cohomology ring $H^*(P_\nbar;\bz)$.
\begin{thm}\label{intcoh} Let $\nbar=(n_1,\ldots,n_r)$ with $n_1\le n_i$ for all $i$. Let $n_1=2m_1+\eps$ with $\eps\in\{0,1\}$.
Let $E=\{i>1:\ n_i\text{ even}\}$. There are classes $z$ with $|z|=2$ and $x_{n_i}$ with $|x_{n_i}|=n_i$ and an isomorphism
of graded rings
$$H^*(P_\nbar;\bz)\approx(A\oplus B\oplus C)\otimes\Lambda[x_{n_i}:i>1,\ n_i\text{ odd}]$$
with
\begin{eqnarray*}A&=&(\bz[z]/(2z,z^{m_1+1}))\otimes\bigl\langle\prod_{j\in S}x_{n_j}:S\subset E,\ |S|\text{ even} \bigr\rangle,\\
B&=&\bigl\langle2x_{n_1}\prod_{j\in S}x_{n_j}:S\subset E,\ |S|\not\equiv\eps\ (2)\bigr\rangle,\\
C&=&(\zt[z]/(z^{m_1+\eps}))\otimes\bigl\langle p_S:S\subset E,\ |S|\text{ odd}\bigr\rangle,
\end{eqnarray*}
where $|p_S|=1+\sum\limits_{j\in S}n_j$ and if $i,j\in E-S$, then $x_{n_i}x_{n_j}p_S=p_{S\cup\{i,j\}}$.
The only nontrivial products are the indicated products by $x_{n_i}$'s or by $z$. \end{thm}
Here $|S|$ denotes the cardinality of the set $S$, and $\langle-\rangle$ denotes the span of a set of elements.
Note that if $n_i$ is even, $x_{n_i}$ is not in $H^*(P_\nbar;\bz)$, but is involved in various product expressions.
The reason for the factor 2 in the terms in $B$ is to denote their image under $H^*(P_\nbar)\to H^*(S^{n_1}\times\cdots\times S^{n_r})$.
\begin{proof} As in the proof of \ref{fldcoef}, we use the Serre spectral sequence of $S^{n_1}\times\cdots\times S^{n_r}\to P_\nbar\to RP^\infty$.
The spectral sequence is a sum of two forms, and these vary with the parity of $n_1$.  For every product $\prod\limits_{i\in T} x_{n_i}$ with $T$ a subset of $\{2,\ldots,r\}$, there is a portion of the spectral sequence as in one
of the four diagrams below. In these diagrams, a dot represents $\zt$. Once these portions of the spectral sequence are noted,
the conclusion of the theorem is mostly just bookkeeping.
\end{proof}
\medskip
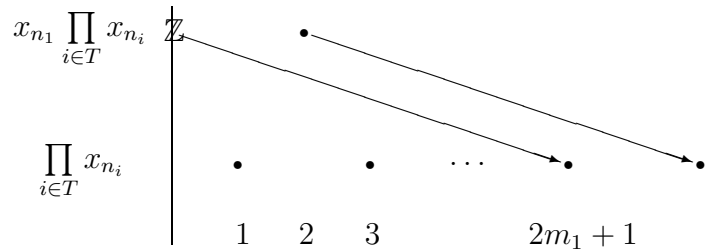
\begin{minipage}{6.5in}
\begin{diag}\label{evev}{$n_1=2m_1$, $|T\cap E|$ even, yielding $A$ in \ref{intcoh}}
\begin{center}
\begin{picture}(430,90)(100,0)
\def\mp{\multiput}
\def\elt{\circle*{3}}
\put(200,0){\line(0,1){90}}
\put(197,27){$\bz$}
\mp(250,30)(50,0){2}{\elt}
\mp(375,30)(50,0){2}{\elt}
\mp(225,80)(50,0){2}{\elt}
\mp(228,79)(50,0){2}{\vector(3,-1){144}}
\put(224,0){$1$}
\put(248,0){$2$}
\put(273,0){$3$}
\put(298,0){$4$}
\put(360,0){$2m_1+2$}
\put(330,30){$\ldots$}
\put(150,29){$\prod\limits_{i\in T}x_{n_i}$}
\put(140,79){$x_{n_1}\prod\limits_{i\in T}x_{n_i}$}
\end{picture}
\end{center}
\end{diag}
\end{minipage}
\bigskip

\bigskip
\begin{minipage}{6.5in}
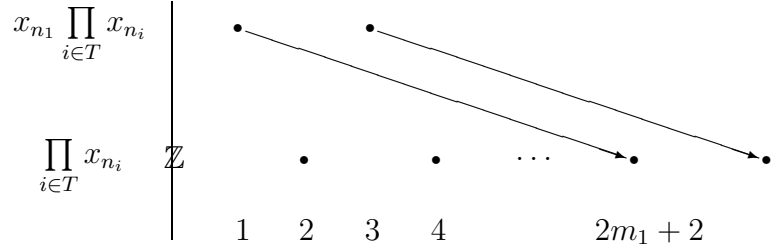
\begin{diag}\label{evod}{$n_1=2m_1$, $|T\cap E|$ odd, yielding $B$ and $C$ in \ref{intcoh}}
\begin{center}
\begin{picture}(430,90)(100,0)
\def\mp{\multiput}
\def\elt{\circle*{3}}
\put(200,0){\line(0,1){90}}
\put(197,77){$\bz$}
\mp(225,30)(50,0){2}{\elt}
\mp(350,30)(50,0){2}{\elt}
\mp(250,80)(50,0){1}{\elt}
\mp(203,79)(50,0){2}{\vector(3,-1){144}}
\put(224,0){$1$}
\put(248,0){$2$}
\put(273,0){$3$}
\put(335,0){$2m_1+1$}
\put(305,30){$\ldots$}
\put(150,29){$\prod\limits_{i\in T}x_{n_i}$}
\put(140,79){$x_{n_1}\prod\limits_{i\in T}x_{n_i}$}
\end{picture}
\end{center}
\end{diag}
\end{minipage}
\bigskip

\bigskip
\begin{minipage}{6.5in}
\begin{diag}\label{odev}{$n_1=2m_1+1$, $|T\cap E|$ even, yielding $A$ and $B$ in \ref{intcoh}}
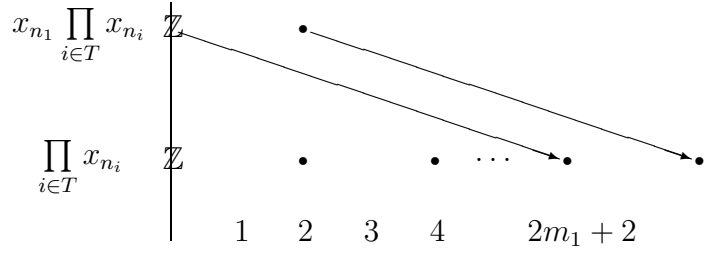
\begin{center}
\begin{picture}(430,90)(100,0)
\def\mp{\multiput}
\def\elt{\circle*{3}}
\put(200,0){\line(0,1){90}}
\mp(197,77)(0,-50){2}{$\bz$}
\mp(250,30)(50,0){2}{\elt}
\mp(350,30)(50,0){2}{\elt}
\mp(250,80)(50,0){1}{\elt}
\mp(203,79)(50,0){2}{\vector(3,-1){144}}
\put(224,0){$1$}
\put(248,0){$2$}
\put(273,0){$3$}
\put(298,0){$4$}
\put(335,0){$2m_1+2$}
\put(315,30){$\ldots$}
\put(150,29){$\prod\limits_{i\in T}x_{n_i}$}
\put(140,79){$x_{n_1}\prod\limits_{i\in T}x_{n_i}$}
\end{picture}
\end{center}
\end{diag}
\end{minipage}
\bigskip

\bigskip
\begin{minipage}{6.5in}
\begin{diag}\label{odod}{$n_1=2m_1+1$, $|T\cap E|$ odd, yielding $C$ in \ref{intcoh}}
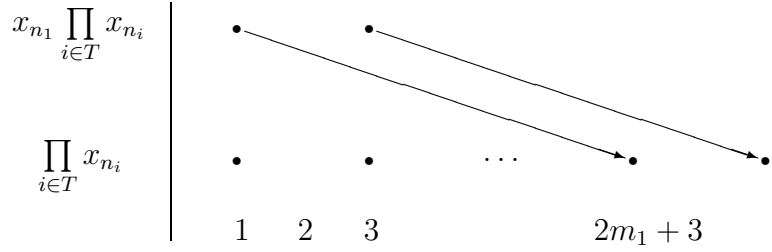
\begin{center}
\begin{picture}(430,90)(100,0)
\def\mp{\multiput}
\def\elt{\circle*{3}}
\put(200,0){\line(0,1){90}}
\mp(225,30)(50,0){2}{\elt}
\mp(375,30)(50,0){2}{\elt}
\mp(225,80)(50,0){2}{\elt}
\mp(228,79)(50,0){2}{\vector(3,-1){144}}
\put(224,0){$1$}
\put(248,0){$2$}
\put(273,0){$3$}
\put(360,0){$2m_1+3$}
\put(318,30){$\ldots$}
\put(150,29){$\prod\limits_{i\in T}x_{n_i}$}
\put(140,79){$x_{n_1}\prod\limits_{i\in T}x_{n_i}$}
\end{picture}
\end{center}
\end{diag}
\end{minipage}
\medskip

Theorem \ref{intcoh} suggests the possibility that the $S^{n_i}$ with $n_i$ odd might be product factors of $P_\nbar$.
The following result shows the limited extent to which this is true. Here and throughout, $\nu(-)$ denotes the exponent of 2 in an integer,
and $\phi(n)$ is the number of positive integers $\le n$ which are congruent to $0$, $1$, $2$, or $4$ mod $8$.
\begin{thm}\label{he} Let $n_1\le n_i$ for all $i$. Let $T=\{i>1:\ \nu(n_i+1)\ge\phi(n_1)\}$.
\begin{enumerate}
\item If $\mbar$ denotes the subtuple of $\nbar$ whose subscripts are not in $T$, then there is a homeomorphism $P_\nbar\approx P_\mbar\times\prod\limits_{i\in T}S^{n_i}$.
\item If $i\not\in T$ and $\mbar$ is obtained from $\nbar$ by omitting $n_i$, then there does not exist a homotopy equivalence $P_\nbar\simeq P_\mbar\times  S^{n_i}$.
    \end{enumerate}
    \end{thm}
\begin{proof} (1). For each $i\in T$, there is an action of $S^{n_1}$ on $S^{n_i}$ via Clifford modules (\cite{ABS}).
The homeomorphism
$$S^{n_1}\times\cdots\times S^{n_r}\quad \mapright{\prod f_i}\quad S^{n_1}\times\cdots\times S^{n_r}$$
defined by
$$f_i(x_1,\ldots,x_r)=\begin{cases} x_i&i\not\in T\\
x_1\cdot x_i&i\in T\end{cases}$$
passes to the desired homeomorphism
$$P_\mbar\times\prod_{i\in T}S^{n_i}\to P_\nbar.$$

(2) Assume such an equivalence exists, and, without loss of generality, that $i=2$. Using \ref{splitthm}, its suspension is an equivalence
\begin{eqnarray*}&&\bigvee_{\ubar\subset(n_2,\ldots,n_r)}\Sigma^{1-\ell(\ubar)}P_{|\ubar|+\ell(\ubar)}^{n_1+|\ubar|+\ell(\ubar)}\\
&\simeq&S^{n_2+1}\vee\bigvee_{\ubar\subset(n_3,\ldots,n_r)}\biggl(\Sigma^{1-\ell(\ubar)}P_{|\ubar|+\ell(\ubar)}^{n_1+|\ubar|+\ell(\ubar)}
\vee\Sigma^{2+n_2-\ell(\ubar)}P_{|\ubar|+\ell(\ubar)}^{n_1+|\ubar|+\ell(\ubar)}\biggr).
\end{eqnarray*}
The wedge summands common to both sides, $\bigvee\limits_{\ubar\subset(n_3,\ldots,n_r)}\Sigma^{1-\ell(\ubar)}P_{|\ubar|+\ell(\ubar)}^{n_1+|\ubar|+\ell(\ubar)}$,
might have some common summands of $S^{n_2+1}$, but  the only way that the LHS could have one in addition to those is if the bottom
cell of $P_{n_2+1}^{n_1+n_2+1}$ splits off, and this happens iff $\nu(n_2+1)\ge\phi(n_1)$.
\end{proof}

Next we determine the ring $K^*(P_\nbar)$. Here $K^*(-)$ denotes unreduced $\zt$-graded periodic complex $K$-theory.
The result is quite similar to that for integral cohomology.
\begin{thm}\label{Kthm} Let $\nbar$, $m_1$, $\eps$, and $E$ be as in \ref{intcoh}. There is an isomorphism of
$\zt$-graded rings,
$$K^*(P_\nbar)\approx(A_K\oplus B_K\oplus C_K)\otimes\Lambda[x_{n_i}:\ i>1,\ n_i\text{ odd}],$$
with
\begin{eqnarray*}A_K&=&(\bz\oplus\bz/2^{m_1})\otimes\bigl\langle\prod_{j\in S}x_{n_j}:\ S\subset E,\ |S|\text{ even}\bigr\rangle\\
B_K&=&\bigl\langle 2x_{n_1}\prod_{j\in S}x_{n_j}:\ S\subset E,\ |S|\not\equiv\eps\ (2)\bigr\rangle\\
C_K&=&\bigl\langle p_S:\ S\subset E,\ |S|\text{ odd}\bigr\rangle/2^{m_1+\eps}.
\end{eqnarray*}
Here $x_{n_j}\in K^{n_j}(P_\nbar)$ with the superscript of $K$ being considered mod 2. Also, $A_K\subset K^0(P_\nbar)$, $B_K\subset K^\eps(P_\nbar)$, and $C_K\subset K^1(P_\nbar)$. The notation for $C_K$ means that each $p_S$ has order $2^{m_1+\eps}$. Note that $S=\emptyset$ is allowed in $A_K$, and so $A_K$
contains the initial $\bz\oplus\bz/2^{m_1}$, whose generators are 1 and $g:=(1-\xi)$, satisfying $g^2=2g$.
If $i,j\in E-S$, then $x_{n_i}x_{n_j}p_S=p_{S\cup \{i,j\}}$. We have $g\cdot p_S=2p_S$. Other than obvious products with $x_{n_i}$'s, there
are no other nontrivial products.
\end{thm}
\begin{proof} This mostly follows from the Atiyah-Hirzebruch spectral sequence (AHSS)
$$H^*(P_\nbar;K^*(pt))\Rightarrow K^*(P_\nbar).$$
We use the description of $H^*(P_\nbar;\bz)$ in \ref{intcoh} together with Diagrams \ref{evev}, \ref{evod}, \ref{odev}, and
\ref{odod}. We use Theorem \ref{splitthm} and the well-known result for $K^*(P_n^k)$ (\cite{vfld}) to see that there are no differentials in the AHSS
and that the $\zt$'s along a row in one of the diagrams extend cyclically.
Most of the product structure can be seen in the AHSS, including the product $g\cdot p_S$.

A more $K$-theoretic way to see $g\cdot p_S=2p_S$ can be obtained using Theorem \ref{splitthm}. First note that
$p_S$ can be interpreted as an element of $K^0(\Sigma P_\nbar)$, and, using \ref{splitthm}, as an element
of $K^0(\Sigma^{1-\ell(\ubar)}P_{|\ubar|+\ell(\ubar)}^{|\ubar|+\ell(\ubar)+n_1})$ with $\ell(\ubar)$ odd, for
$\ubar$ corresponding to $S$. Then $p_S$ corresponds to $2^{(|\ubar|+\ell(\ubar)-1)/2}(1-\xi_{|\ubar|+\ell(\ubar)+n_1})
\in K(P_{|\ubar|+\ell(\ubar)}^{|\ubar|+\ell(\ubar)+n_1})$. There is an action of $K(P^{n_1})$ on $K(P_{|\ubar|+\ell(\ubar)}^{|\ubar|+\ell(\ubar)+n_1})\approx K(T(|\ubar|+\ell(\ubar))\xi_{n_1})$ using the
action of $K(D(\theta))$ on $K(T(\theta))$, and an action of $K(P^{n_1})$ on $K(\Sigma P_\nbar)$ using the projection map $P_\nbar\to P^{n_1}$. Using (\ref{explicit}) and (\ref{suspended}), one can show that the isomorphism of \ref{splitthm} is
compatible with these actions. Thus $g\cdot p_S$ corresponds to $(1-\xi)\cdot 2^e(1-\xi)=2(2^e(1-\xi))$, for
appropriate $e$.
\end{proof}

\section{Manifold properties}\label{mfsec}
Two properties of manifolds $M$ studied by algebraic topologists are $\spa(M)$ and $\imm(M)$, defined by
\begin{defin}\label{defs} If $M$ is a differentiable manifold, $\spa(M)$ is the maximal number of linearly
independent tangent vector fields on $M$, while $\imm(M)$ is the dimension of the smallest Euclidean space
in which $M$ can be immersed.\end{defin}
\noindent In this section, we study $\spa(P_\nbar)$ and $\imm(P_\nbar)$. We also determine exactly when $P_\nbar$ is parallelizable.

The answers are related to $\spa(k\xi_n)$ and $\gd(\ell\xi_n)$, where $k$ is a positive integer, $\ell$ an integer,
and $\xi_n$ the Hopf bundle over real projective space $P^n$. These much-studied quantities are defined in
\begin{defin}\label{gddef} If $\theta$ is a vector bundle over a topological space $X$, $\spa(\theta)$ is the
maximal number of linearly independent (l.i.) sections of $\theta$. If $\eta$ is a stable vector bundle over $X$,
then its geometric dimension, $\gd(\eta)$, is the smallest $k$ such that there is a $k$-plane bundle over $X$
stably equivalent to $\eta$.\end{defin}

The following facts relating these concepts are well-known.
\begin{prop}\begin{itemize}
\item If $M$ is a manifold with tangent bundle $\tau(M)$, then $\spa(M)=\spa(\tau(M))$.
\item If $\theta$ is a $d$-dimensional vector bundle over a finite-dimensional CW complex $X$ with $d>\dim(X)$,
then $\gd(\theta)+\spa(\theta)=d$.
\item If $m>0$,  $\nu(L)\ge\phi(n)$, and $L-m>n$, then $$\gd(-m\xi_n)=\gd((L-m)\xi_n)=L-m-\spa((L-m)\xi_n).$$

\end{itemize}\label{easy}
\end{prop}

The immersion dimension of $P_\nbar$ is related to the geometric dimension of a stable vector bundle over
a projective space in the following result. It seems somewhat strange that $\imm(P_\nbar)$ does not depend
on the values of most of the $n_i$.
\begin{thm} If $\nbar=(n_1,\ldots,n_r)$ with $n_1\le n_i$ for all $i$, then
$$\imm(P_\nbar)=|\nbar|+\max(\gd(-(|\nbar|+r)\xi_{n_1}),1).$$
\label{immthm}\end{thm}
\begin{proof} The tangent bundle $\tau(P_\nbar)$ is given by
$$\{(\ubar,\xbar)\in\br^{n_1+1}\times\cdots\times\br^{n_r+1}\times S^{n_1}\times\cdots\times S^{n_r}: u_i\perp x_i\ \forall i\}/((\ubar,\xbar)\sim(-\ubar,-\xbar)),$$
with $\ubar=(u_1,\ldots,u_r)$ and $\xbar=(x_1,\ldots,x_r)$.
There is a vector bundle isomorphism
\begin{equation}\label{tang}\tau(P_\nbar)\oplus r\vareps\mapright{\approx}(|\nbar|+r)\xi_\nbar\end{equation}
defined by
$$([u_1,\ldots,u_r,x_1,\ldots,x_r],t_1,\ldots,t_r)\mapsto[u_1+t_1x_1,\ldots,u_r+t_rx_r,\xbar].$$
Here and throughout $r\vareps$ denotes a trivial bundle of dimension $r$, and $t_i\in\br$.

The maps $$P^{n_1}\mapright{j}P_\nbar \mapright{p} P^{n_1}$$
defined by $j([x])=[x,\ldots,x]$ and $p([x_1,\ldots,x_r])=[x_1]$ satisfy \begin{equation}\label{xis}p^*(\xi_{n_1})=\xi_\nbar\quad\text{and}\quad
j^*(\xi_\nbar)=\xi_{n_1}.\end{equation}  By \cite{Hirsch}, $\imm(P_\nbar)$ equals $|\nbar|$ plus the geometric dimension
of the stable normal bundle of $P_\nbar$, unless this gd is 0, in which case $\imm(P_\nbar)=|\nbar|+1$. By (\ref{tang}), the stable normal bundle of $P_\nbar$ is
$-(|\nbar|+r)\xi_\nbar$. By (\ref{xis}), we have $\gd(-(|\nbar|+r)\xi_\nbar)\le\gd(-(|\nbar|+r)\xi_{n_1})$ and
$\gd(-(|\nbar|+r)\xi_{n_1})\le\gd(-(|\nbar|+r)\xi_{\nbar})$, implying the result.
\end{proof}

Since, by obstruction theory, if $\eta$ is a stable vector bundle over a CW complex $X$, then $\gd(\eta)\le \dim(X)$, we obtain the
following surprising corollary.
\begin{cor} If $n_1\le n_i$ for all $i$, then $P_\nbar$ can be immersed in $\br^{|\nbar|+n_1}$.\label{immcor}
\end{cor}

An immediate corollary of (\ref{tang}) is
\begin{cor} $P_\nbar$ is orientable if and only if $|\nbar|+r$ is even.\end{cor}

Geometric dimension of multiples of the Hopf bundle over real projective spaces, sometimes called the
generalized vector field problem, has been studied in many papers such as \cite{Ad}, \cite{Annals}, \cite{thesis}, \cite{Lam}, and \cite{LR}.
One consequence of Theorem \ref{immthm} is that every case of the generalized vector field problem is solving an immersion question
for some manifold.
In Section \ref{numsec}, we combine specific results on the generalized vector field problem with Theorem \ref{immthm} to obtain numerical bounds on $\imm(P_\nbar)$
for certain $\nbar$.

Our second manifold result, involving $\spa(P_\nbar)$, is similar, but not quite so complete.
It is better expressed in terms of stable span, defined for a manifold $M$ by $\ssp(M)= \spa(\tau(M)+\vareps)-1$. It is a well-known
consequence of obstruction theory that if $r\ge1$, then
$\spa(\tau(M)+r\vareps)-r$ is independent of $r$, and hence equals  $\ssp(M)$.
Clearly $\spa(M)\le\ssp(M)$.
\begin{thm}\begin{enumerate}
\item If $\nbar=(n_1,\ldots,n_r)$ with $n_1\le n_i$ for all $i$, then
\begin{equation}\label{spanineq}\ssp(P_\nbar)=\spa((|\nbar|+r)\xi_{n_1})-r.\end{equation}
\item $\spa(P_\nbar)=0$ if and only if all $n_i$ are even.
\item If $|\nbar|$ is even but not all $n_i$ are even, then $\spa(P_\nbar)=\ssp(P_\nbar)$.

\item If $|\nbar|\equiv3\pmod8$ and $r\equiv 1\pmod4$, then $\spa(P_\nbar)=\ssp(P_\nbar)$.
\end{enumerate}\label{spanthm}
\end{thm}
\begin{proof} \begin{enumerate}
\item Both sides of (\ref{spanineq}) equal $\spa((|\nbar|+r)\xi_\nbar)-r$,
one side using (\ref{tang}) and the other using (\ref{xis}).

\item This is immediate from the classical theorem of Hopf that $$\spa(M)>0\text{ iff }\chi(M)=0,$$ together
with the fact, from \ref{cohthm}, that the Euler characteristic $\chi(P_\nbar)=\frac12\prod(1+(-1)^{n_i})$.

\item
Koschorke showed in \cite[Theorem 20.1]{Kosh} that if $\dim(M)$ is even and $\chi(M)=0$, then $\spa(M)=\ssp(M)$. As just noted, $\chi(P_\nbar)=0$ if and only if some $n_i$ is odd. This part of the theorem
is now immediate from Koschorke's result.

\item Koschorke also showed in \cite[Corollary 20.10]{Kosh} that if $M$ is a Spin manifold with $\dim(M)\equiv3$ mod 8 and $\chi_2(M)=0$, then
$\spa(M)=\ssp(M)$. Here $\chi_2(M)$ is the Kervaire semicharacteristic, defined, for odd-dimensional manifolds, as the mod 2 value of the sum of
the ranks of the even-dimensional mod-2 homology groups. Using Theorem \ref{cohthm}, one easily shows that if $|\nbar|$ is odd, then $\chi_2(P_\nbar)=0$ unless
$r=1$ and $n_1\equiv1$ mod 4, or $r=2$, $n_1$ is even, and $n_2$ odd.
Note that we needed $|\nbar|+r\equiv 0$ mod 4 in order that $P_\nbar$ be a Spin-manifold.
\end{enumerate}
\end{proof}

Similarly to Corollary \ref{immcor}, we have
\begin{cor} If $n_1\le n_i$ for all $i$, then $\ssp(P_\nbar)\ge |\nbar|-n_1$.\end{cor}

A closely-related result tells exactly when $P_\nbar$ is parallelizable. Here $\nu$ and $\phi$ are as defined prior to Theorem \ref{he}.
\begin{thm}\label{pithm} If $\nbar=(n_1,\ldots,n_r)$ with $n_1\le n_i$ for all $i$, then $P_\nbar$ is parallelizable if and only if
$\nu(|\nbar|+r)\ge\phi(n_1)$ and not all $n_i$ are even.\end{thm}
\begin{proof} Bredon and Kosinski proved in \cite{BK} that a stably parallelizable $n$-manifold $M$ is parallelizable if and only if
$n$ is even and $\chi(M)=0$ or $n$ is odd and $\chi_2(M)=0$. By (\ref{tang}) and (\ref{xis}), $\tau(P_\nbar)$ is stably trivial
iff $(|\nbar|+r)\xi_{n_1}$ is, and this is true iff $\nu(|\nbar|+r)\ge\phi(n_1)$. The theorem now follows from our observations about
$\chi(P_\nbar)$ and $\chi_2(P_\nbar)$ in the proof of \ref{spanthm}.
\end{proof}

An approach to showing that span equals stable span for an $n$-manifold $M$ was presented in \cite{JT}.
In the exact sequence
$$[\Sigma M, BO]\mapright{\delta} [M,V_n]\to [M,BO(n)] \to [M,BO],$$
$[M,V_n]$ has two elements, with the nontrivial element being detected in $\zt$-cohomology. This is the cause of the
possibility of there being an element in $[M,BO(n)]$ stably equivalent to the tangent bundle but not equal to it.
If there is an element $\a$ in $[\Sigma M,BO]=\widetilde{KO}(\Sigma M)$ such that $\delta(\a)\ne0$, then we can deduce
that span equals stable span. Such an element $\a$ is specified in \cite{JT} by a condition on its Stiefel-Whitney classes.
The splitting of $\Sigma P_\nbar$ in Theorem \ref{splitthm} enables us to understand $\widetilde{KO}(\Sigma P_\nbar)$.
However, it seems that there are no elements whose Stiefel-Whitney classes satisfy the required condition.

\section{Some numerical results for $\imm(P_\nbar)$ and $\spa(P_\nbar)$}\label{numsec}
In this section, we sample some of the known results about $\gd(k\xi_n)$ and discuss their implications for $P_\nbar$.

Using Stiefel-Whitney classes and construction of blinear maps, Lam proved the following result in \cite[Thm 1.1]{Lam}.
\begin{prop} \label{Lamprop} $\gd(k\xi_n)\ge m_0$, where $m_0$ is the largest $m\le n$ for which $\binom km$ is odd.
Equality occurs here if $\binom{[k/8]}{[n/8]}$ is odd.\end{prop}

The following well-known proposition is often useful in determining whether binomial coefficients are odd.
\begin{prop} If $k=2^{e_0}+\cdots+2^{e_t}$ with $e_0<\cdots<e_t$, let $\Bin(k)=\{e_0,\ldots,e_t\}$.
If $k>0$, then $\binom kn$ is odd iff $\Bin(n)\subseteq\Bin(k)$, and $\binom{-k}n$ is odd iff $\Bin(k-1)$ and $\Bin(n)$ are disjoint.
\end{prop}

Note that equality occurs in Proposition \ref{Lamprop} if $n\le 7$, and, using \ref{immthm}, we easily obtain the following corollary.
\begin{cor} If $\nbar=(n_1,\ldots,n_r)$ with $n_1\le n_i$ for all $i$, and $n_1\le 7$, then $\imm(P_\nbar)=|\nbar|+n_1-\delta$,
where $\delta$ is given in Table \ref{tab1}.\end{cor}

\medskip
\begin{minipage}{6.5in}
\begin{tab}\label{tab1}{Values of $\delta$}
\begin{center}
\begin{tabular}{cc|cccccccc}
&&&\multicolumn{6}{c}{$|\nbar|+r\pmod8$}\\
&&$1$&$2$&$3$&$4$&$5$&$6$&$7$&$8$\\
\hline
&$0$&$0$&$0$&$0$&$0$&$0$&$0$&$0$&$0$\\
$n_1$&$1$&$0$&$1$&$0$&$1$&$0$&$1$&$0$&$1$\\
&$2$&$0$&$0$&$1$&$2$&$0$&$0$&$1$&$2$\\
&$3$&$0$&$1$&$2$&$3$&$0$&$1$&$2$&$3$\\
&$4$&$0$&$0$&$0$&$0$&$1$&$2$&$3$&$4$\\
&$5$&$0$&$1$&$0$&$1$&$2$&$3$&$4$&$5$\\
&$6$&$0$&$0$&$1$&$2$&$3$&$4$&$5$&$6$\\
&$7$&$0$&$1$&$2$&$3$&$4$&$5$&$6$&$7$
\end{tabular}
\end{center}
\end{tab}
\end{minipage}
\medskip

Combining triviality of $16\xi_8$ and $32\xi_9$ with results in \cite{Lam}, we also have complete information
about $\gd(k\xi_8)$ and $\gd(k\xi_9)$, which we state in Proposition \ref{89}. These, with \ref{immthm}, yield complete
information about $\imm(P_\nbar)$ when $n_1=8$ or 9. It is quite remarkable that whenever the smallest subscript
$n_1$ is $\le 9$, the immersion dimension of $P_\nbar$ is precisely known.
Other results about $\imm(P_\nbar)$ can be obtained by combining \ref{Lamprop} and \ref{immthm}, but we will
not bother to state them.
\begin{prop}\label{89} If $0\le\Delta\le 15$ and $i\ge0$, then $\gd((16i+\Delta)\xi_8)=\min(\Delta,8)$.
If $i\ge0$, then
$$\gd((16i+\Delta)\xi_9)=\begin{cases}\Delta&\text{if $i$ even and $0\le\Delta\le9$, or $i$ odd and $\Delta=6$ or $7$}\\
9&\text{if $\Delta=9$, $11$, $13$, or $15$}\\
8&\text{if $\Delta=8$, $10$, $12$, or $14$}\\
6&\text{if $i$ is odd and $\Delta=0$, $2$, $3$, or $5$}\\
5&\text{if $i$ is odd and $\Delta=1$ or $4$.}
\end{cases}$$
\end{prop}

The implications of Proposition \ref{Lamprop} for $\spa(P_\nbar)$ are limited by the span-versus-stablespan conundrum.
We readily obtain the following result about stable span.
\begin{prop} Let $\nbar=(n_1,\ldots,n_r)$ with $n_1\le n_i$ for all $i$. Then $\ssp(P_\nbar)\le|\nbar|-m_1$, where $m_1$ is the largest $m\le n_1$
such that $\binom{|\nbar|+r}m$ is odd. Equality is obtained if $\binom{[(|\nbar|+r)/8]}{[n_1/8]}$ is odd. If $n_1\le7$, then
$\ssp(P_\nbar)=|\nbar|-n_1+\delta$, where $\delta$ is as in Table \ref{tab1} with column $i$ replaced by $8-i$; i.e., the column labels
read $7,\ldots,0$.\label{ssprop}\end{prop}

Combining Theorem \ref{spanthm} with Proposition \ref{ssprop} yields results about $\spa(P_\nbar)$. However,
even if $n_1\le7$, we must be careful about trying to assert $\spa(P_\nbar)=\ssp(P_\nbar)$ because of the situation described in part
(2) of Theorem \ref{spanthm}.

The second type of geometric dimension result on which we focus is vector bundles of low geometric dimension.
These were first studied by Adams in \cite{Ad}, and Lam and Randall provide the current status in \cite{LR},
some of which is described in the following theorem.
\begin{thm}\label{ALR} $($\cite{Ad},\cite{LR}$)$ Assume $n\ge18$.
\begin{itemize}
\item If $0\le d\le 4$, then $\gd(k\xi_n)=d$ if and only if
$k\equiv d$ mod $2^{\phi(n)}$.
\item If $\nu(k)=\phi(n)-1$  and $n\not\equiv7$ mod $8$ or if $\nu(k)=\phi(n)-2$
and $n\equiv2$ or $4$ mod $8$, then $\gd(k\xi_n)=5$.
\item The only other possible occurrences of $\gd(k\xi_n)=5$ for $k\equiv0$ mod $4$ are $(\nu(k)=\phi(n)-1$ and $n\equiv7$ mod $8)$ or $(\nu(k)=\phi(n)-2$
and $n\equiv1,3,5$ mod $8)$.
\end{itemize}
\end{thm}

This has the following immediate consequence for $\imm(P_\nbar)$.
\begin{cor} Let $\nbar=(n_1,\ldots,n_r)$ with $n_1\le n_i$ for all $i$.
\begin{itemize}
\item $\imm(P_\nbar)=|\nbar|+1$ iff $\nu(|\nbar|+r)\ge\phi(n_1)$ or $\nu(|\nbar|+r+1)\ge\phi(n_1)$.
\item For $2\le d\le 4$, $\imm(P_\nbar)=|\nbar|+d$ iff $\nu(|\nbar|+r+d)\ge\phi(n_1)$.
\item If $\nu(|\nbar|+r)=\phi(n_1)-1$ and $n_1\not\equiv7$ mod $8$, or if $\nu(|\nbar|+r)=\phi(n_1)-2$ and
$n_1\equiv2,4$ mod $8$, then $\imm(P_\nbar)=|\nbar|+5$.
\end{itemize}
\end{cor}

Results such as
$$\text{``if }0\le d\le4,\text{ then } \ssp(P_\nbar)=|\nbar|-d\text{ iff }\nu(|\nbar|+r-d)\ge\phi(n_1)\text{"}$$
can also be immediately read off from \ref{spanthm} and \ref{ALR}.

Next we recall the implications of Adams operations in $K$-theory for sectioning $k\xi_n$. Although slightly
stronger results can be obtained using $KO$-theory, we prefer here the following simpler-to-state $KU$ result.
\begin{thm}\label{Kgd}$($\cite{thesis}$)$ If $\binom{m-1}n$ is odd, then $m\xi_n$ has at most $m-n+2\nu(m)+1$
l.i. sections.\end{thm}

The implication of this for $P_\nbar$ is given in the following result, which is immediate from \ref{immthm}, \ref{spanthm}, and \ref{Kgd}
\begin{cor} Let $\nbar=(n_1,\ldots,n_r)$ with $n_1\le n_i$ for all $i$.
\begin{itemize}
\item If $\binom{-|\nbar|-r-1}{n_1}$ is odd, then $\imm(P_\nbar)\ge|\nbar|+n_1-2\nu(|\nbar|+r)-1$.
\item If $\binom{|\nbar|+r-1}{n_1}$ is odd, then $\spa(P_\nbar)\le|\nbar|-n_1+2\nu(|\nbar|+r)+1$.
\end{itemize}
\end{cor}

Finally, we recall the strong implications of $BP$-theory for sectioning $k\xi_n$.
Slightly stronger results have been recently obtained using $tmf$ (\cite{tmf}) or $ER(2)$ (\cite{KW}),
but we list here the $BP$-result because it is much simpler to state.
\begin{thm} $($\cite{Annals}$)$ If $\nu\binom{n+s}{k-s}=s$, then $\gd(2n\xi_{2k})\ge 2k-6s$.
\end{thm}

In applying this, it is useful to note that $\nu\binom{\ell+m}\ell=\a(\ell)+\a(m)-\a(\ell+m)$, where
$\a(m)$ is the number of 1's in the binary expansion of $m$. This implies that $\nu\binom{2\ell+2m}{2\ell}=\nu\binom
{\ell+m}{\ell}$, which we will use in the next result. The implications for $P_\nbar$ are as follows, derived
in the usual way.
\begin{cor} Assume $|\nbar|+r$ and $n_1$ are even. Then
\begin{itemize}
\item If $\nu\binom{-|\nbar|-r+2s}{n_1-2s}=s$, then $\imm(P_\nbar)\ge |\nbar|+n_1-6s$.
\item If $\nu\binom{|\nbar|+r+2s}{n_1-2s}=s$, then $\spa(P_\nbar)\le|\nbar|-n_1+6s$.
\end{itemize}
\end{cor}

\def\line{\rule{.6in}{.6pt}}

\end{document}